\documentclass[12pt]{amsart}

\usepackage{begnac}
\usepackage{stmaryrd}

\DeclareMathOperator{\BiPert}{BiPert}

\title{On perturbations of continuous structures}

\author{Itaï \textsc{Ben Yaacov}}

\address{Itaï \textsc{Ben Yaacov} \\
  Université Claude Bernard -- Lyon 1 \\
  Institut Camille Jordan \\
  43 boulevard du 11 novembre 1918 \\
  69622 Villeurbanne Cedex \\
  France}

\urladdr{\url{http://math.univ-lyon1.fr/~begnac/}}

\thanks{Research partially supported by NSF grant DMS-0500172,
  ANR chaire d'excellence junior THEMODMET (ANR-06-CEXC-007) and
  by Marie Curie research network ModNet.}
\thanks{The author would like to thank
  the Isaac Newton Institute and the organisers of the programme on
  Model Theory and Applications to Algebra and Analysis, during which
  this work was initiated;
  C.\ Ward Henson for many helpful discussions and comments;
  and Hernando Tellez for a careful reading of the manuscript.}

\svnInfo $Id: Perturb.tex 865 2009-03-27 12:55:26Z begnac $
\thanks{\textit{Revision} {\svnInfoRevision} \textit{of} \today}


\begin{document}

\begin{abstract}
  We give a general framework for the treatment of perturbations of
  types and structures in continuous logic, allowing to
  specify which parts of the logic may be perturbed.
  We prove that separable, elementarily equivalent structures which are
  approximately $\aleph_0$-saturated up to arbitrarily small
  perturbations are isomorphic up to arbitrarily small perturbations
  (where the notion of perturbation is part of the data).
  As a corollary, we obtain a Ryll-Nardzewski style characterisation
  of complete theories all of whose separable models are isomorphic
  up to arbitrarily small perturbations.
\end{abstract}

\maketitle

\section*{Introduction}

In this paper we define what we call perturbation systems and study
their basic properties.
These are objects which formalise the intuitive notion of
allowing chosen parts of a metric structure
to be perturbed by arbitrarily
small amounts.

One motivation for this notion comes from an unpublished result of
C.\ Ward Henson,
consisting of a Ryll-Nardzewski style characterisation of complete
continuous theories of pure Banach spaces which are
separably categorical up to
arbitrarily small perturbation of the norm
(but not of the underlying linear structure).
Seeking a general framework in which such results can be proved,
we develop a general formalism for the consideration of
metric structures and types up to small perturbations,
which gives rise in particular to a notion of categoricity up to
perturbation.
In \fref{thm:PertRN} we give a general
Ryll-Nardzewski style characterisation of
complete countable continuous theories which are separably categorical
up to arbitrarily small perturbation, where the precise notion of
perturbation is part of the given data alongside the theory.
One convenient way of specifying a ``perturbation system'' $\fp$
is via the perturbation distance $d_\fp$ between types,
where $d_\fp(p,q) \in [0,\infty]$ measures by how much a model needs
to be perturbed so that a realisation of $p$ may become a realisation of
$q$ (and $d_\fp(p,q) = \infty$ if this is impossible).

Our criterion for $\aleph_0$-categoricity up to perturbation bears
considerable resemblance to the one used by Henson, as
both criteria compare the standard logic topology on a space of types
with an appropriate metric arising from the perturbation system.
In Henson's criterion, the topology is compared directly to the
Banach-Mazur perturbation distance $d_{BM}$
on the space of types of \emph{linearly independent} tuples
of a Banach space, which he calls $\tS_n^*$.
In the general case considered in \fref{thm:PertRN}
we do not have an analogue
of $\tS_n^*$, so the comparison must take place on the entire type
space.
This entails an additional complexity, not present in Henson's
criterion, in that the topology must be compared to
an appropriate combination of the perturbation metric $d_\fp$
with the standard distance $d$.
A result based on Henson's criterion appears in a
subsequent paper \cite{BenYaacov:Unbounded},
where we deal with further complications caused by the fact
that a Banach space
is an unbounded structure whose unit ball is not preserved by
a non trivial Banach-Mazur perturbation.

A second motivation comes from some open problems concerning the
automorphism group of the separable model of an $\aleph_0$-categorical
continuous theory.
Such problems could be addressed from a model-theoretic point of view
as questions concerning the theory $T_A$ (i.e., $T$ with a generic
automorphism, or even several non commuting ones).
Just as the underlying metric of a continuous structure induces a
natural metric on the space of types, it also induces one on its
automorphism group, namely the metric of uniform convergence
(if the structure is discrete, so are the induced metrics,
so they simply do not arise as interesting objects in the classical
discrete setting).
The model theoretic counterpart of the consideration of
small metric neighbourhoods of an automorphism is the consideration of
$(M,\sigma) \models T_A$ up to small perturbations of $\sigma$.
While the present paper does not contain any results in this
direction, this did serve well as an example towards the
general setting, and in fact was at the origin of the
author's interest in perturbations.

A common feature of these two instances is that only
part of the structure is allowed to be perturbed
while the rest is kept untouched.
In the first case, the norm is perturbed while the linear structure is
untouched, while in the second it is only the
automorphism that we perturb (and not the original structure).
Thus a ``notion of perturbation'' should say what parts of the
structure can be perturbed, and in what way.
Also, in order to state a Ryll-Nardzewski style result concerning
perturbations we need to consider on the one hand perturbations of
(separable) models, and on the other perturbations of types.

\medskip

In \fref{sec:Pert} we compare these two notions (perturbations of
structures and of types): requiring them to be compatible yields the
notion of a \emph{perturbation radius}.
In order to speak of ``arbitrarily small perturbation'' we need to
consider a system of perturbation radii decreasing to the zero
perturbation, which with some natural additional properties yields the
notion of a \emph{perturbation system}.

In \fref{sec:PertOSat} we study a variant of the notion of
approximate $\aleph_0$-saturation which takes into account a perturbation
system, and show that separable models which are saturated in this
sense are also isomorphic up to small perturbation.

In \fref{sec:PertRN} we prove the main result (\fref{thm:PertRN}) and
discuss various directions in which it may and may not be further generalised
(\fref{thm:OpenPertRN} vs.\ \fref{exm:LpCntrExm}).

In \fref{sec:PertAut} we
conclude with a few questions concerning perturbations of
automorphisms.

\medskip

Notation is mostly standard.
We use $a$, $b$, $c$, \ldots to denote members of structures,
and use $x$, $y$, $z$, \ldots to denote variables.
Bar notation is used for (usually finite) tuples, and uppercase
letters are used for sets.
We also write $\bar a \in A$ to say that $\bar a$ is a tuple
consisting of members of $A$, i.e., $\bar a \in A^n$ where
$n = |\bar a|$.

We work in the
framework of continuous first order logic, as developed in
\cite{BenYaacov-Usvyatsov:CFO}.
Most of the time we work within the context of a fixed continuous
theory $T$ in a language $\cL$.
We always assume that $T$ is closed under logical consequences.
In particular, $|T| = |\cL| + \aleph_0$ and
$T$ is countable if and only if $\cL$ is.

For a general survey of the model theory of metric structures
we refer the reader to
\cite{BenYaacov-Berenstein-Henson-Usvyatsov:NewtonMS}.

\section{Perturbations}
\label{sec:Pert}

\subsection{Perturbation pre-radii}

We start by formalising the notion of allowing structures and types
to be perturbed ``by this much''.
We start by defining perturbation pre-radii, which tell us which types
can be changed into which:
\begin{dfn}
  \label{dfn:PertPreRadius}
  A \emph{perturbation pre-radius} $\rho$ (for a fixed theory $T$)
  is a family of closed
  subsets $\{\rho_n \subseteq \tS_n(T)^2\}$ containing the diagonals.
  If $X \subseteq \tS_n(T)$, then the $\rho$-neighbourhood around $X$ is
  defined as:
  \begin{align*}
    X^\rho = \{q \colon (\exists p\in X) \, (p,q) \in \rho_n\}.
  \end{align*}
  Notice that if $X$ is closed then so is $X^\rho$.
\end{dfn}

\begin{rmk}
  A predecessor of sorts to this definition exists in José Iovino's
  notion of a \emph{uniform structure} on the type spaces of a positive
  bounded theory \cite{Iovino:StableBanach}.
  Specifically, a uniform structure in Iovino's sense can be generated
  by vicinities which are given by perturbation pre-radii (although
  the definition of a pre-radius does not appear in Iovino's work).
  Specific perturbation systems (see \fref{dfn:PertSys} below)
  of importance, such as the Banach-Mazur system,
  occur as uniform structures in Iovino's work. 
\end{rmk}

We wish to consider mappings which perturb structures: they need not be
elementary, and are merely required to respect the
perturbation pre-radius.
\begin{dfn}
  \begin{enumerate}
  \item 
    Let $\rho$ be a perturbation pre-radius, $M,N \models T$.
    A \emph{partial $\rho$-perturbation} from $M$ into $N$
    is a partial mapping $f\colon M \dashrightarrow N$ such that for every
    $\bar a \in \dom(f)$:
    \begin{align*}
      \tp^N(f(\bar a)) \in \tp^M(\bar a)^\rho.
    \end{align*}
    If $f$ is total then it is a $\rho$-perturbation of $M$
    into $N$.
    The set of all $\rho$-perturbations of $M$ into $N$ is
    denoted $\Pert_\rho(M,N)$.
  \item If $f \in \Pert_\rho(M,N)$ is bijective, and
    $f^{-1} \in \Pert_\rho(N,M)$, we say that $f\colon M\to N$ is a
    \emph{$\rho$-bi-perturbation}, in symbols $f \in \BiPert_\rho(M,N)$.
  \item   We say that two perturbation pre-radii $\rho$ and $\rho'$ are
    \emph{equivalent}, in symbols $\rho \sim \rho'$,
    if $\Pert_\rho(M,N) = \Pert_{\rho'}(M,N)$ for all
    $M,N \models T$.
    We say they are \emph{bi-equivalent}, in symbols $\rho \approx \rho'$,
    if $\BiPert_\rho(M,N) = \BiPert_{\rho'}(M,N)$ for all
    $M,N \models T$.

    Note that $\rho \sim \rho' \Longrightarrow \rho \approx \rho'$.
  \item If $\rho$ and $\rho'$ are two perturbation pre-radii,
    we write $\rho \leq \rho'$ to mean that $\rho_n \subseteq \rho'_n$
    for all $n$ (i.e., $\rho$
    is \emph{stricter} than $\rho'$).
  \end{enumerate}
\end{dfn}

\begin{lem}
  For every perturbation pre-radius $\rho$ there exists a minimal
  perturbation pre-radius equivalent to $\rho$, denoted
  $\langle\rho\rangle$, and a minimal perturbation pre-radius bi-equivalent
  to $\rho$, denoted $\llbracket\rho\rrbracket$.

  If $\rho = \langle\rho\rangle$ we say that $\rho$ is \emph{reduced}.
  If $\rho = \llbracket\rho\rrbracket$ we say that $\rho$ is \emph{bi-reduced}.
\end{lem}
\begin{proof}
  One just verifies that
  $\langle\rho\rangle = \bigcap \{\rho'\colon \rho' \sim \rho\}$ and $\llbracket\rho\rrbracket = \bigcap \{\rho'\colon \rho' \approx \rho\}$ are perturbation
  pre-radii which are equivalent and bi-equivalent, respectively, to
  $\rho$.
\end{proof}

Note that $\llbracket\rho\rrbracket \leq \langle\rho\rangle \leq \rho$, so if $\rho$ is bi-reduced it is reduced.

\begin{dfn}
  Let $\rho,\rho'$ be perturbation pre-radii.
  We define their \emph{composition}
  as the pre-radius $\rho' \circ \rho$
  defined by:
  \begin{align*}
    (\rho' \circ \rho)_n = \{(p,q)\colon \exists r\, (p,r) \in \rho_n \text{ and } (r,q) \in \rho'_n\}.
  \end{align*}
\end{dfn}
It may be convenient to think of a perturbation pre-radius as the
graphs of a family multi-valued mappings $\rho_n\colon \tS_n(T) \to \tS_n(T)$.
In this case, our notion of composition above is indeed the
composition of multi-valued mappings.

Notice that we also obtain a composition mapping for perturbations:
\begin{gather*}
  \circ\colon \Pert_\rho(M,N) \times \Pert_{\rho'}(N,L) \to \Pert_{\rho\circ\rho'}(M,L).
\end{gather*}

The minimal perturbation pre-radius is
$\id = \{\id_n\colon n \in \bN\}$, where $\id_n$ is the diagonal of $\tS_n(T)$,
i.e., the graph of the identity mapping.
It is bi-reduced, $\rho \circ \id = \id \circ \rho = \rho$ for all $\rho$,
and an $\id$-perturbation is
synonymous with an elementary embedding.

\subsection{Perturbation radii}

A perturbation pre-radius imposes a family of
conditions saying which types may be perturbed to which.
We may further require these conditions to be compatible with one
another:
\begin{dfn}
  A \emph{perturbation radius} is a pre-radius $\rho$ satisfying that
  for any two types $(p,q) \in \rho_n$ there exist models
  $M$ and $N$ and a $\rho$-perturbation $f\colon M \to N$ sending some
  realisation of $p$ to a realisation of $q$.
\end{dfn}
Notice that the identity perturbation pre-radius is a perturbation
radius.

We now try to break down the notion of a perturbation radius into
several technical properties and see what each of them means.

For our purposes, a \emph{(uniform) continuity modulus} is
a mapping $\delta\colon (0,\infty) \to (0,\infty)$
which is increasing and left-continuous.
(In other words, this is a mapping satisfying
$\delta(\varepsilon)
= \sup_{\varepsilon'<\varepsilon} \delta(\varepsilon')$.
This additional property does not play any role at this stage,
but is harmless to assume.)
A mapping between metric spaces
$f\colon (X,d) \to (X',d')$ \emph{respects} $\delta$ if for all
$\varepsilon > 0$
and all $x,y \in X$:
\begin{gather*}
  d(x,y)< \delta(\varepsilon)
  \quad \Longrightarrow \quad
  d'(f(x),f(y)) \leq \varepsilon.
\end{gather*}
Such a mapping $f$ is uniformly continuous if and only if it respects
some uniform continuity modulus.

\begin{dfn}
  Let $\rho$ be a perturbation pre-radius.
  \begin{enumerate}
  \item We say that $\rho$ \emph{respects equality}
    if
    $$[x=y]^\rho = [x=y].$$
    (I.e., if $(p,q) \in \rho_2$,
    $p \models x=y$, then $q \models x = y$ as well).
  \item We say that $\rho$ \emph{respects a continuity modulus $\delta$} if
    every $\rho$-perturbation does.
  \item We say that $\rho$ is \emph{uniformly continuous} if it respects
    some continuity modulus $\delta$.
  \item We say that $\rho$ respects a continuity modulus $\delta$
    \emph{trivially} if for all $\varepsilon > 0$:
    \begin{gather*}
      [d(x,y) < \delta(\varepsilon)]^\rho \subseteq [d(x,y) \leq \varepsilon].
    \end{gather*}
  \end{enumerate}
\end{dfn}

\begin{lem}
  \label{lem:RespEq}
  A perturbation pre-radius $\rho$ respects equality if and only if
  there exists a continuity modulus $\delta$ which $\rho$ respects trivially.
\end{lem}
\begin{proof}
  Assume first that $\rho$ respects $\delta$ trivially.
  For every $\varepsilon > 0$ we have $\delta(\varepsilon) > 0$, whereby
  $[x=y] \subseteq [d(x,y) < \delta(\varepsilon)]$ and thus
  $[x=y]^\rho \subseteq [d(x,y) \leq \varepsilon]$.
  Therefore $[x=y]^\rho \subseteq \bigcap_{\varepsilon>0} [d(x,y)\leq\varepsilon] = [x=y]$.
  As the other inclusion is always true, we obtain equality.

  Conversely, assume that $\rho$ respects no $\delta$ trivially.
  Then there exists some $\varepsilon > 0$ such that for all $\delta > 0$
  there is some pair
  $(p_\delta,q_\delta) \in \rho_2$ such that
  $p_\delta \in[d(x,y)<\delta]$ and $q_\delta \in [d(x,y) > \varepsilon]$.
  Since $\tS_2(T)^2$ is compact this sequence has an accumulation
  point $(p,q)$ as $\delta$ goes to $0$.
  Since $\rho_2$ is closed we have $(p,q) \in \rho_2$, and clearly
  $(p,q) \in [x=y]\times[d(x,y) \geq \varepsilon]$ as well,
  so $[x=y]^\rho \neq [x=y]$.
\end{proof}

\begin{lem}
  Let $\rho$ be a perturbation pre-radius
  and $\delta$ a continuity modulus.
  If $\rho$ respects $\delta$ trivially then it respects $\delta$.
  Conversely, if $\rho$ respects $\delta$ then $\langle\rho\rangle$
  respects $\delta$ trivially.

  In particular, if $\rho$ is reduced then it respects $\delta$ if and only if
  it respects it trivially.
\end{lem}
\begin{proof}
  The first statement is straightforward.
  For the converse, let:
  \begin{align*}
    X_\delta
    = & \bigl\{
    (p,q) \in \tS_2(T)\colon
    (\forall\varepsilon>0)
    (d(x,y)^p < \delta(\varepsilon) \to d(x,y)^q\leq\varepsilon)
    \bigr\} \\
    = & \bigcap_{\varepsilon>0} \bigl([d(x,y) \geq \delta(\varepsilon)]\times\tS_2(T) \cup \tS_2(T) \times [d(x,y) \leq\varepsilon] \bigr).
  \end{align*}
  Let $\rho'$ be obtained from $\rho$ by replacing
  $\rho_2$ with $\rho_2 \cap X_\delta$.
  Notice that the identity mapping of any model of
  $T$ is a $\rho$-perturbation and must therefore respect $\delta$, so $X_\delta$
  contains the diagonal and $\rho'$ is a perturbation pre-radius.
  Clearly $\rho' \sim \rho$, so $\langle\rho'\rangle = \langle\rho\rangle$, and $\rho'$ respects $\delta$
  trivially, whereby so does $\langle\rho\rangle$.
\end{proof}

Note that if $\rho$ is uniformly continuous,
$M,N \models T$, and $A \subseteq M$, then any partial $\rho$-perturbation
$f\colon A \to M$ is uniformly continuous.
It therefore extends uniquely to a mapping
$\bar f\colon \bar A \to M$.
As $\rho$ is given by closed sets, the completion $\bar f$ is also
a $\rho$-perturbation.

\begin{lem}
  \label{lem:UCRed}
  A perturbation pre-radius $\rho$ is a perturbation radius if and only
  if it is uniformly continuous and reduced.
\end{lem}
\begin{proof}
  Left to right is easy.
  For right to left, consider the family
  $\cF = \{(f,M,N)\colon M,N \models T, f \in \Pert_\rho(M,N)\}$.
  Since $\rho$ is uniformly continuous, ultra-products of families of
  triplets in $\cF$ exist, and since $\rho$ as a perturbation pre-radius
  consists of closed sets, $\cF$ is closed under ultra-products.
  Define $\rho_\cF$ by:
  $$\rho_{\cF,n} = \{(\tp(\bar a),\tp(\bar b))\colon
  (f,M,N) \in \cF, \bar a \in M^n, \bar b = f(\bar a)\}.$$
  Since $\cF$ is closed under ultra-products and contains all the
  identity mappings, $\rho_\cF$ is a perturbation pre-radius.
  It clearly satisfies $\rho_{\cF} \leq \rho$, $\rho_\cF \sim \rho$, and as $\rho$ is
  reduced we conclude that $\rho_\cF = \rho$.

  On the other hand, it is clear from the construction of $\rho_\cF$ that
  it is a perturbation radius.
\end{proof}

\begin{prp}
  \label{prp:PertRadGen}
  A perturbation pre-radius $\rho$ is equivalent to a perturbation radius
  if and only if it is uniformly continuous, in which case
  $\langle\rho\rangle$ is the unique perturbation radius equivalent to $\rho$.
\end{prp}
\begin{proof}
  Immediate from \fref{lem:UCRed}.
\end{proof}

Recall that if $p(\bar x,\bar y)$ is a partial type, then the property
$\exists\bar y\,p(\bar x,\bar y)$ (where the existential quantifier varies
over a sufficiently saturated elementary extension) is also definable
by a partial type.
\begin{dfn}
  A perturbation pre-radius $\rho$ \emph{respects the existential
    quantifier $\exists$}
  if for every
  partial type $p(\bar x,\bar y)$:
  $$[\exists\bar y \, p(\bar x,\bar y)]^\rho = [\exists\bar y\,p^\rho(\bar x,\bar y)].$$
\end{dfn}

\begin{lem}
  A perturbation pre-radius $\rho$ respects $\exists$ if and only if
  for every two sufficiently saturated models $M,N \models T$,
  tuples $\bar a \in M^n$, $\bar b \in N^n$, and $c \in M$:
  $$\tp(\bar b) \in \tp(\bar a)^\rho \Longleftrightarrow
  (\exists d \in N) \bigl( \tp(\bar bd) \in \tp(\bar ac)^\rho \bigr).$$
\end{lem}
\begin{proof}
  Easy.
\end{proof}

If $\sigma$ is an $n$-permutation, it acts on $\tS_n(T)$
by
$\sigma^*(p(x_{<n})) = p(x_{\sigma^{-1}(0)},\ldots,x_{\sigma^{-1}(n-1)})$
(so
$\sigma^*(\tp(a_{<n})) = \tp(a_{\sigma(0)},\ldots,a_{\sigma(n-1)})$).
\begin{dfn}
  A perturbation pre-radius $\rho$ is \emph{permutation-invariant}
  if for every $n$, and every permutation $\sigma$ on $n$ elements, $\rho_n$
  is invariant under the action of $\sigma$.
  In other words, for every $p,q \in \tS_n(T)$:
  $$(p,q) \in \rho_n \Longleftrightarrow (\sigma^*(p),\sigma^*(q)) \in \rho_n.$$
\end{dfn}

\begin{prp}
  \label{prp:PertRadResp}
  Let $\rho$ be a perturbation pre-radius.
  Then the following are equivalent:
  \begin{enumerate}
  \item $\rho$ is a perturbation radius.
  \item $\rho$ respects $=$, $\exists$ and is permutation-invariant.
  \item Whenever $M,N \models T$, $\bar a \in M^n$, $\bar b \in N^n$,
    and $\tp(\bar b) \in \tp(\bar a)^\rho$,
    there exist an elementary extension $N' \succeq N$ and a
    $\rho$-perturbation $f\colon M \to N'$ sending $\bar a$
    to $\bar b$.
  \end{enumerate}
\end{prp}
\begin{proof}
  \begin{cycprf}
  \item[\impnext] Straightforward.
  \item[\impnext] Let $\bar a \in M$ and $\bar b \in N$ be such that
    $\bar b \models \tp(\bar a)^\rho$.
    Let $N' \succeq N$ realise every type of finite tuples over finite
    tuples in $N$.

    Since $\rho$ respects $\exists$ and $N'$ is sufficiently saturated,
    for every $\bar c \in M$ there is $\bar d \in N'$ such that
    $\bar b\bar d \models \tp(\bar a\bar c)^\rho$.
    Since $\rho$ respects equality, $(a_i \mapsto b_i) \cup (c_j \mapsto d_j)$ is a
    well-defined mapping, call it $f\colon M \dashrightarrow N'$.
    Since $\rho$ respects $\exists$ and is permutation-invariant,
    $f$ is a partial $\rho$-perturbation.

    Let $I$ be the family of all partial $\rho$-perturbations
    $f\colon M \dashrightarrow N'$ where $\dom(f)$ is finite containing
    $\bar a$, and $f(\bar a) = \bar b$.
    For any tuple $\bar c \in M$, let
    $J_{\bar c} = \{f \in I\colon \bar c \subseteq \dom(f)\}$, and let $\cF \subseteq \cP(I)$
    be the filter generated by the $J_{\bar c}$.
    By the argument above $\cF$ is a proper filter, and therefore
    extends to an ultra-filter $\sU$.

    Let $N'' = {N'}^\sU$.
    Let $g\colon M\to N''$ be given by $g = \prod_{f\in I} f/\sU$.
    In other words, for every $c \in M$ we define $g(c) \in N''$ to be
    $[c_f\colon f \in I] \in N''$, where $c_f = f(c)$ if $c \in \dom(f)$:
    since $J_{\{c\}}$ is a large set, we need not care about $c_f$ for other
    values of $f$.
    Identifying $N'$ with its diagonal embedding in $N''$ we have
    $N \preceq N' \preceq N''$, and clearly $g(\bar a) = \bar b$.

    Finally, for every (finite) tuple $\bar c \in M$ we have
    $g(\bar c) = [\bar c_f\colon f \in I]$, where $\bar c_f = f(\bar c)$ for
    every $f$ in the large set $J_{\bar c}$.
    Since $\tp(\bar c_f) \in \tp(\bar c)^\rho$ for all $f \in J_{\bar c}$,
    and $\tp(\bar c)^\rho$ is a closed set, we must have
    $\tp(g(\bar c)) \in \tp(\bar c)^\rho$.

    We conclude that $g\colon M \to N''$ is a $\rho$-perturbation as required.
  \item[\impfirst] Clear.
  \end{cycprf}
\end{proof}

It follows that the composition of perturbation radii is again one:
\begin{lem}
  \label{lem:Comp}
  If $\rho,\rho'$ are perturbation radii then $\rho' \circ \rho$
  is a perturbation
  radius as well.
\end{lem}
\begin{proof}
  Assume that $q \in p^{\rho' \circ \rho}$.
  Then there is a type $r \in p^\rho$ such that $q \in r^{\rho'}$.
  Let $\bar a \models p$ in $M$.
  Then there is a model $N$ and $\rho$-perturbation
  $f\colon M \to N$ such that $f(\bar a) \models r$, and a
  $\rho'$-perturbation
  $g\colon N \to L$ such that $g \circ f(\bar a) \models q$.
\end{proof}

Recall from \cite{BenYaacov:PositiveModelTheoryAndCats} that the
\emph{type-space functor}
of $T$ is a contra-variant functor from $\bN$ to topological spaces,
sending an object $n \in \bN$ to $\tS_n(T)$, and a mapping
$\sigma\colon n \to m$ to the mapping
\begin{gather*}
  \begin{array}{cccc}
    \sigma^* \colon & \tS_m(T) & \to & \tS_n(T) \\
    & \tp(a_i\colon i<m) & \mapsto & \tp(a_{\sigma(i)}\colon i<n).
  \end{array}
\end{gather*}

We obtain the following elegant characterisation of perturbation radii:
\begin{lem}
  \label{lem:PertRadFunct}
  A perturbation pre-radius $\rho$ is a perturbation radius
  if and only if for every $n,m \in \bN$
  and mapping $\sigma\colon n \to m$, the induced mapping
  $\sigma^* \colon \tS_m(T) \to \tS_n(T)$ satisfies
  that for all $p \in \tS_m(T)$:
  $$\sigma^*(p^\rho) = \sigma^*(p)^\rho.$$
  Viewing $\rho$ as the family of graphs of multi-valued mappings, we could
  write this property more simply as $\sigma^* \circ \rho_m = \rho_n \circ \sigma^*$.
  Thus a perturbation pre-radius is a perturbation radius
  if and only if it commutes with the type-space functor structure
  on $\{\tS_n(T)\colon n \in \bN\}$.
\end{lem}
\begin{proof}
  Assume that $\rho$ is a perturbation radius, and let
  $\sigma\colon n \to m$ be a mapping.
  Let $p \in \tS_m(T)$, $q \in \tS_n(T)$, and let $a_{<m} \in M$ realise
  $p$.
  Then each of $q \in \sigma^*(p^\rho)$ and
  $q \in \sigma^*(p)^\rho$ is equivalent to the existence of
  a $\rho$-perturbation $g\colon M \to N$ such that
  $q = \tp(g(a_{\sigma(i)})\colon i < n)$.

  Conversely, assume that $\sigma^*(p^\rho) = \sigma^*(p)^\rho$ for all
  $\sigma\colon n\to m$ and $p \in \tS_m(T)$.
  When restricted to
  the special case where $\sigma\colon 2\to 1$ is the unique such mapping,
  this is equivalent to $\rho$ preserving equality;
  when restricted to the family of inclusions $n \hookrightarrow n+1$, this is
  equivalent to $\rho$ preserving $\exists$;
  and when restricted to the permutations of the natural numbers, this
  is equivalent to $\rho$ being permutation-invariant.
  Therefore $\rho$ is a perturbation radius by
  \fref{prp:PertRadResp}.
\end{proof}

\begin{dfn}
  We say that a perturbation radius (or pre-radius) is
  \emph{symmetric} if $q \in p^\rho \Longleftrightarrow p \in q^\rho$.
\end{dfn}

\begin{lem}
  \label{lem:SymPertExt}
  Assume that $\rho$ is a symmetric perturbation radius,  and let
  $f \in \Pert_\rho(M,N)$.
  Then there exist elementary extensions
  $M' \succeq M$, $N' \succeq N$, and a bi-perturbation
  $f' \in \BiPert(M',N')$ extending $f$.
\end{lem}
\begin{proof}
  Since $\rho$ is symmetric then 
  $f^{-1}\colon f(M) \to M$ is a partial $\rho$-perturbation,
  and since $\rho$ is a perturbation radius,
  we may extend $f^{-1}$ to a
  $\rho$-perturbation $g\colon N \to M' \succeq M$.
  Proceeding this way
  we may thus construct two elementary chains $(M_i\colon i \in \bN)$ and
  $(N_i\colon i \in \bN)$ such that $M_0 = M$, $N_0 = N$, and two sequences of
  $\rho$-perturbations $f_i\colon M_i \to N_i$ and $g_i\colon N_i \to M_{i+1}$ such that
  $f_0 = f$, $g_i\circ f_i = \id_{M_i}$, and $f_{i+1}\circ g_i = \id_{N_i}$.
  Then at the limit we obtain $M_\omega \succeq M$ and $N_\omega \succeq N$, $\rho$-perturbations
  $f_\omega\colon M_\omega \to N_\omega$ and $g_\omega\colon N_\omega \to M_\omega$ such that $g_\omega = f_\omega^{-1}$.
  Thus every $\rho$-perturbation can be extended by a back-and-forth 
  argument to a $\rho$-bi-perturbation $f_\omega \in \BiPert_\rho(M_\omega,N_\omega)$.
\end{proof}

\begin{lem}
  \label{lem:UCBiRed}
  A perturbation pre-radius $\rho$ is a symmetric
  perturbation radius if and only
  if it is uniformly continuous and bi-reduced.
\end{lem}
\begin{proof}
  If $\rho$ is bi-reduced then it is reduced and symmetric, so one
  direction is by \fref{lem:UCRed}.
  For the other, assume $\rho$ is a symmetric perturbation radius.
  Let $f \in \Pert_\rho(M,N)$, and let
  $f' \in \BiPert_\rho(M',N')$ extend it as in \fref{lem:SymPertExt}.
  Then $f' \in \BiPert_{\llbracket\rho\rrbracket}(M',N')$ by definition, whereby
  $f' \in \Pert_{\llbracket\rho\rrbracket}(M',N')$ and $f \in \Pert_{\llbracket\rho\rrbracket}(M,N)$.
  Therefore $\llbracket\rho\rrbracket \sim \rho$, and as both are reduced they are equal.
\end{proof}

\begin{prp}
  \label{prp:SymPertRadGen}
  A perturbation pre-radius $\rho$ is bi-equivalent to a
  symmetric perturbation radius
  if and only if it is uniformly continuous, in which case
  $\llbracket\rho\rrbracket$ is the unique symmetric perturbation radius bi-equivalent to
  $\rho$.
\end{prp}
\begin{proof}
  Immediate from \fref{lem:UCBiRed}.
\end{proof}

Finally, we may find the following observation useful:
\begin{lem}
  \label{lem:SymPertRadGenComp}
  Let $\rho^i$ be symmetric uniformly continuous
  perturbation pre-radii such that $\rho^1 \circ \rho^0 \leq \rho^2$.
  Then $\llbracket\rho^1\rrbracket \circ \llbracket\rho^0\rrbracket \leq \llbracket\rho^2\rrbracket$.
\end{lem}
\begin{proof}
  Let $M_0,M_1 \models T$ and $f \in \Pert_{\llbracket\rho^1\rrbracket \circ \llbracket\rho^0\rrbracket}(M_0,M_1)$.
  Then every restriction of $f$ to a finite set
  can be decomposed by definition
  into a partial $\llbracket\rho^0\rrbracket$-perturbation followed by a
  partial $\llbracket\rho^1\rrbracket$-perturbation.
  We can glue these together by an ultra-product argument
  to obtain $M_i' \succeq M_i$ for $i < 2$ and $M_2' \models T$,
  such that $f$ extends to
  $f' \colon M_0' \to M_1'$, which in turn decomposes into a
  $\llbracket\rho^0\rrbracket$-perturbation $g'\colon M_0' \to M_2'$ followed by
  a $\llbracket\rho^1\rrbracket$-perturbation $h'\colon M_2' \to M_1'$.

  Since $\llbracket\rho^i\rrbracket$ are symmetric perturbation radii,
  we may use a back-and-forth argument as in the proof
  of \fref{lem:SymPertExt}
  to construct extensions
  $M_i'' \succeq M_i'$ for $i < 3$ and
  $g'' \in \BiPert_{\llbracket\rho^0\rrbracket}(M''_0,M''_2)$,
  $h'' \in \BiPert_{\llbracket\rho^1\rrbracket}(M''_2,M''_1)$.
  It follows that
  $f'' = h'' \circ g'' \in \Pert_{\rho^2}(M''_0,M_1'')$ is bijective.
  Since $\rho^2$ is assumed to be symmetric, $f''$ is a
  $\rho^2$-bi-perturbation and therefore a
  $\llbracket\rho^2\rrbracket$-perturbation.

  This shows that $\llbracket\rho^1\rrbracket \circ \llbracket\rho^0\rrbracket \leq \llbracket\rho^2\rrbracket$.
\end{proof}

\subsection{Perturbation systems}

A single perturbation radius gives us certain leverage at perturbing
types.
But our goal is not to study perturbations by a single perturbation
radius,
but rather by ``arbitrarily small'' perturbation radii, where the
notion of a small perturbation radius depends on the context.
We formalise this through the notion of a perturbation system:

Let $\fR^0$ denote the family of perturbation pre-radii, and
$\fR$ denote the family of perturbation radii.

\begin{dfn}
  \label{dfn:PertSys}
  A \emph{perturbation pre-system} is a mapping
  $\fp \colon \bR^+ \to \fR^0$ satisfying:
  \begin{enumerate}
  \item \emph{Downward continuity:} If $\varepsilon_n \searrow \varepsilon$ then $\fp(\varepsilon) = \bigcap \fp(\varepsilon_n)$.
  \item \emph{Symmetry:} $\fp(\varepsilon)$ is symmetric for all $\varepsilon$.
  \item \emph{Triangle inequality:} $\fp(\varepsilon) \circ \fp(\varepsilon') \leq \fp(\varepsilon + \varepsilon')$.
  \item \emph{Strictness:} $\fp(0) = \id$.
  \end{enumerate}
  If in addition its range lies in $\fR$, then
  $\fp\colon \bR^+ \to \fR$ is a \emph{perturbation system}.
\end{dfn}

Given a perturbation (pre-)system $\fp$, we may define the
\emph{perturbation distance} between two types $p,q \in \tS_n(T)$ as:
\begin{gather*}
  d_{\fp,n}(p,q) = d_{\fp}(p,q) = \inf \{\varepsilon \geq 0\colon (p,q) \in \fp_n(\varepsilon)\}.
\end{gather*}
Notice that by strictness and the triangle inequality this is indeed a
$[0,\infty]$-valued metric, where infinite distance means that neither
type can be perturbed into the other.

\begin{lem}
  \label{lem:PertSysMet}
  Let $\fp$ be a perturbation pre-system.
  Then the family of metrics $(d_{\fp,n}\colon n \in \bN)$ has the following
  properties:
  \begin{enumerate}
  \item For every $n$, the set
    $\{(p,q,\varepsilon) \in \tS_n(T)^2 \times \bR^+\colon d_{\fp,n}(p,q) \leq \varepsilon\}$ is closed.
  \item If $\fp$ is a perturbation system, then
    for every $n,m \in \bN$ and mapping $\sigma\colon n \to m$, the induced
    mapping $\sigma^* \colon \tS_m(T) \to \tS_n(T)$ satisfies
    for all $p \in \tS_m(T)$ and $q \in \tS_n(T)$:
    $$d_{\fp,m}(p,(\sigma^*)^{-1}(q)) = d_{\fp,n}(\sigma^*(p),q).$$
    (Here we follow the convention that $d_{\fp,m}(p,\emptyset) = \inf \emptyset = \infty$.)
  \end{enumerate}
  Conversely, given a family of metrics with values in $[0,\infty]$
  satisfying the first property,
  and defining
  $\fp_n(\varepsilon) =
  \{(p,q) \in \tS_n(T)^2\colon d_{\fp,n}(p,q) \leq \varepsilon\}$,
  we obtain that $\fp$ is a perturbation pre-system, and it is a
  perturbation system if and only if
  the second property is satisfied as well.
\end{lem}
\begin{proof}
  This is merely a reformulation:\\
  -- Symmetry, triangle inequality and strictness correspond to each
  $d_{\fp,n}$ being a metric;\\
  -- Downward continuity corresponds to the set
  $\{(p,q,\varepsilon) \in \tS_n(T)^2 \times \bR^+
  \colon d_{\fp,n}(p,q) \leq \varepsilon\}$
  being closed; and\\
  -- Each of the $\fp(\varepsilon)$ being a perturbation radius
  corresponds to $d_{\fp,m}(p,(f^*)^{-1}(q)) = d_{\fp,n}(f^*(p),q)$,
  by \fref{lem:PertRadFunct}.
\end{proof}

We say that two perturbation systems $\fp$ and $\fp'$ are
\emph{equivalent} if the perturbation metrics $d_{\fp}$ and
$d_{\fp'}$ are uniformly equivalent on each $\tS_n(T)$.

We say that a perturbation pre-system $\fp$
\emph{respects equality}
if $\fp(\varepsilon)$ does for all $\varepsilon > 0$.
In this case, by \fref{prp:SymPertRadGen} we can define
$\llbracket\fp(\varepsilon)\rrbracket
= \llbracket\fp\rrbracket(\varepsilon)$
to be the symmetric perturbation radius generated by $\fp(\varepsilon)$.
By \fref{lem:SymPertRadGenComp}, $\llbracket\fp\rrbracket$ satisfies
the triangle
inequality.
One can verify that $\llbracket\fp\rrbracket$ satisfies downward
continuity, and it is
clearly symmetric and strict, so it is a perturbation system.
As expected, we call $\llbracket\fp\rrbracket$ the perturbation system
\emph{generated} by $\fp$.

\subsection{A few natural examples (and a non-example)}
\label{sec:PertExm}

If $\cL$ consists of finitely many predicate symbols, a natural
perturbation system for $\cL$ is the one allowing to perturb all
symbols by ``a little''.
In order to construct it we first define a perturbation pre-system
$\fp$ by letting $\fp(\varepsilon)$ be the (symmetric) perturbation
pre-radius
allowing the distance symbol $d$ to change by
a multiplicative factor of $e^{\pm\varepsilon}$, and every other symbol to
change by $\pm\varepsilon$.
Then $\fp$ respects equality, and thus generates a perturbation
system $\llbracket\fp\rrbracket$.
Similarly, if $\cL$ is an expansion of $\cL_0$ by finitely many
symbols, we might want to require that all symbols of $\cL_0$
be preserved
precisely, while allowing the new symbols to be perturbed as in the
previous case.

A particularly interesting example of the latter kind is the case of
adding a generic automorphism to a stable continuous theory.
Consider for example the case of infinite dimensional Hilbert spaces:
If $\sigma,\sigma' \in U(H)$, then $(H,\sigma')$ is obtained from
$(H,\sigma)$ by a
small perturbation of the automorphism (which keeps the underlying
Hilbert space unmodified) if and only if the operator norm
$\|\sigma - \sigma'\|$ is small.
Thus the notion of perturbation brings into the realm of model theory
the uniform convergence topology on automorphism groups of
structures.
We shall say a little more about this in the last section.

In case of a classical (i.e., discrete)
first order theory $T$ in a finite
language, there are no non trivial perturbation systems.
Indeed, let $\fp$ be a perturbation system, and
let $P$ be an $n$-ary predicate symbol.
Let
$$X_P = ([P(\bar x)] \times [\neg P(\bar y)])
\cup ([\neg P(\bar x)] \times [P(\bar y)]) \subseteq \tS_n(T)^2.$$
Then $X_P \cap \fp_n(0) = \emptyset$, but $X$ is compact, so there is
$\varepsilon_P > 0$
such that $X_P \cap \fp_n(\varepsilon_P) = \emptyset$.
Replacing function symbols with their graphs we may assume the
language is purely relational, and as we assumed the language to be
finite we have can define
$\varepsilon_0 = \min \{\varepsilon_P\colon P \in \cL\} > 0$.
By the construction every $\fp(\varepsilon_0)$-perturbation is an
elementary
mapping, that is to say that a small enough perturbation, according to
$\fp$, is not a perturbation at all.
Another way of sating this is that $\fp$ is equivalent to the identity
perturbation system.
In short, structures in a finite discrete language cannot really be
perturbed.
The same argument holds if we have a pair of languages
$\cL_0 \subseteq \cL$,
where we only allow to perturb symbols in $\cL \setminus \cL_0$
which are finite in number.

Thus, the notion of perturbation is a new feature of
continuous logic which essentially does not exist in discrete logic.

This last statement is of course not 100\% correct, as there was a
finiteness assumption.
Indeed, let $\cL = \{E_i\colon i \in \bN\}$ and let $T$ be the theory
saying that each $E_i$ is an equivalence relation with two equivalence
classes,
and every intersection of finitely many equivalence classes of
distinct $E_i$'s is infinite.
This is a classical example of a theory which is not
$\aleph_0$-categorical, but every restriction of $T$ to a finite
sub-language is.
For $\varepsilon > 0$, let $\fp(\varepsilon)$ be the symmetric
perturbation radius generated by
requiring $E_i$ to be fixed for all $i < 1/\varepsilon$, and
$\fp(0) = \id$.
Then $\fp$ is a perturbation system, and ``a model of $T$ up to a
small $\fp$-perturbation'' is the same as ``a model of $T$
restricted to a finite sub-language''.
Thus $T$ is $\fp$-$\aleph_0$-categorical in the sense of
\fref{sec:PertOSat} below.

\section{Saturation up to perturbation}
\label{sec:PertOSat}

In this section, $\fp$ denotes a perturbation system for a theory $T$.

\begin{ntn}
  \label{ntn:TypDist}
  If $p(x)$ is any partial type and $\varepsilon \geq 0$ then
  $p(x^\varepsilon)$ denotes the partial type
  \begin{gather*}
    \exists x'\, (p(x') \wedge d(x,x') \leq \varepsilon).
  \end{gather*}

  We define $p(x^\varepsilon,y^\delta,\ldots)$ similarly.
  We follow the convention that the metric on finite tuples is the
  supermum metric, so if $\bar x$ is a (finite) tuple of variables
  then $p(\bar x^\varepsilon)$ means $p(x_0^\varepsilon,x_1^\varepsilon,\ldots)$.
\end{ntn}

This notation can (and will) be used in conjunction with previous
notation.
If $p(\bar x)$ is a partial type and $\rho$ a perturbation radius
then $p^\rho(\bar x)$ is also a partial type, so we can make sense of
$p^\rho(\bar x^\varepsilon)$:
$\bar a \models p^\rho(\bar x^\varepsilon)$ if and only if there
exists a tuple $\bar b$
(in a sufficiently saturated model containing $\bar a$) such that
$d(\bar a,\bar b) \leq \varepsilon$ and $\bar b \models p^\rho$.
Similarly, $\bar a \models p(\bar x^\varepsilon)^\rho$ if and only if there are $\bar b$
and $\bar c$ such that $\bar b \models p$, $d(\bar b,\bar c)\leq \varepsilon$, and
$\tp(\bar a) \in \tp(\bar c)^\varepsilon$.
The difference between the two examples is that in the first we first
perturb $p$ and then allow the realisation to move a little, while in
the second we do it the other way around.
Since $\rho$ is uniformly continuous, this does not make much difference,
as for all $\varepsilon>0$ and $\delta = \delta_\rho(\varepsilon) > 0$:
\begin{align*}
  & [p^\rho(\bar x^\delta)] \subseteq [p(\bar x^\varepsilon)]^\rho, \qquad{} [p(\bar x^\delta)]^\rho \subseteq [p^\rho(\bar x^\varepsilon)].
\end{align*}

\begin{dfn}
  A structure $M$ is \emph{$\fp$-approximately $\aleph_0$-saturated}
  if for every finite tuple $\bar a \in M$, type
  $p(x,\bar a) \in \tS_1(\bar a)$ and $\varepsilon > 0$,
  the partial type
  $p^{\fp(\varepsilon)}(x^\varepsilon,\bar a^\varepsilon)$
  is realised in $M$.
\end{dfn}

Notice that when $b \in M$ realises
$p^{\fp(\varepsilon)}(x^\varepsilon,\bar a^\varepsilon)$,
the witnesses may possibly be \emph{outside} $M$.
In other words,
$M \models p^{\fp(\varepsilon)}(b^\varepsilon,\bar a^\varepsilon)$
only means that there exist
$b',\bar a'$ in some elementary extension of $M$
such that $d(b,b'),d(\bar a,\bar a') \leq \varepsilon$
and $\models p^{\fp(\varepsilon)}(b',\bar a')$.

\begin{lem}
  \label{lem:PertOSatTup}
  The definition of $\fp$-approximate $\aleph_0$-saturation, which was
  given in terms of approximate realisation of $1$-types,
  implies the same property for $n$-types, for any natural $n$.
\end{lem}
\begin{proof}
  Let $M$ be $\fp$-approximately $\aleph_0$-saturated.
  We proceed by induction on $n$.
  For $n = 0$ there is nothing to prove, so we assume for $n$ and
  prove for $n+1$.

  So let $p(x_{\leq n},\bar a) \in \tS_{n+1}(\bar a)$ for some finite tuple
  $\bar a \in M$, where $p(x_{\leq n},\bar y)$ is a complete type without
  parameters, and let $\varepsilon > 0$.
  We need to find in $M$ a realisation for
  $p^{\fp(\varepsilon)}(x_{\leq n}^\varepsilon,\bar a^\varepsilon)$.

  First, find $\delta > 0$ such that
  $[d(x,y)\leq\delta]^{\fp(\varepsilon)}
  \subseteq [d(x,y) \leq \varepsilon/2]$,
  so in particular,
  $\delta \leq \varepsilon/2$.
  Let $q(x_{<n},\bar y) =
  p(x_{\leq n},\bar y){\restriction}_{(x_{<n},\bar y)}$.
  By the induction hypothesis we can realise
  $q^{\fp(\delta)}(x_{<n}^\delta,\bar a^\delta)$ in $M$.
  In other words,
  we can find $b_{<n} \in M$ and $b_{<n}',\bar a'$ possibly
  outside $M$ such that
  $d(b_{<n}'\bar a',b_{<n}\bar a) \leq \delta$ and
  $\models q^{\fp(\varepsilon/2)}(b_{<n}',\bar a')$.
  Since $\fp(\varepsilon/2)$ is a perturbation radius
  we can find $b'_n$ (still, possibly outside $M$)
  such that   $\models p^{\fp(\varepsilon/2)}(b_{\leq n}',\bar a')$.
  Thus in particular:
  \begin{gather*}
    \models p^{\fp(\varepsilon/2)}(b_n',b_{<n}^\delta,\bar a^\delta).
  \end{gather*}

  Let $r(x) = \tp(b_n'/b_{<n},\bar a)$.
  Using $\fp$-approximate $\aleph_0$-saturation, find $b_n \in M$ such that
  $\models r^{\fp(\varepsilon/2)}(b_n^{\varepsilon/2},b_{<n}^{\varepsilon/2},\bar a^{\varepsilon/2})$.
  That is to say that there exist $d_{\leq n},\bar c$
  (possibly outside $M$)
  such that
  \begin{align*}
    & d(b_{\leq n}\bar a,d_{\leq n},\bar c) \leq \varepsilon/2,\\
    & \models r^{\fp(\varepsilon/2)}(d_{\leq n},\bar c),\\
    \intertext{From which we conclude that:}
    & \models p^{\fp(\varepsilon/2)}(d_n,d_{<n}^\delta,\bar c^\delta)^{\fp(\varepsilon/2)},\\
    & \models p^{\fp(\varepsilon)}(d_n,d_{<n}^{\varepsilon/2},\bar c^{\varepsilon/2}),\\
    & \models p^{\fp(\varepsilon)}(b_n^{\varepsilon/2},b_{<n}^\varepsilon,\bar a^\varepsilon),\\
    & \models p^{\fp(\varepsilon)}(b_{\leq n}^\varepsilon,\bar a^\varepsilon).
    \qedhere
  \end{align*}
\end{proof}

\begin{rmk}
  \fref{lem:PertOSatTup} can be restated as saying that
  if a structure $M$ is $\fp$-approximately $\aleph_0$-saturated then
  it is still so after the adjunction of the sort of $n$-tuples
  (namely a sort for $M^n$, equipped with the supremum metric).
  It follows that $\fp$-approximate $\aleph_0$-saturation
  is not affected by the adjunction of
  the sort of $\aleph_0$ -tuples (with the metric
  $d\bigl( (a_n)_{n\in\bN},(b_n)_{n\in\bN} \bigr)
  = \sum 2^{-n-1}d(a_n,b_n)$),
  or of any imaginary sort (with the natural metric).
  Thus the following results can be extended to $\aleph_0$-tuples and
  imaginary sorts as well.
\end{rmk}

\begin{lem}
  \label{lem:PertOSatRlsd}
  Assume that $M$ is $\fp$-approximately $\aleph_0$-saturated.
  Then for every finite tuple $\bar a \in M$, type
  $p(\bar x,\bar a) \in \tS_n(\bar a)$ and $\varepsilon > 0$,
  $p^{\fp(\varepsilon)}(\bar x,\bar a^\varepsilon)$ is realised in $M$.
\end{lem}
\begin{proof}
  By \fref{lem:PertOSatTup} we may assume that $x$ and $a$ are
  singletons.

  Let $\varepsilon_i = (1-2^{-i})\varepsilon$, and choose
  $\delta_i > 0$ small enough so that:
  \begin{enumerate}
  \item $\delta_i \leq 2^{-i-2}\varepsilon$.
  \item $[d(x,y)\leq\delta_i]^{\fp(\varepsilon)} \subseteq [d(x,y) \leq 2^{-i}]$.
  \item $[d(x,y)\leq\varepsilon_i]^{\fp(\delta_i)} \subseteq [d(x,y) \leq \varepsilon_i+2^{-i-2}\varepsilon]$.
  \end{enumerate}
  Notice that the second is possible since $\fp(\varepsilon)$ is uniformly
  continuous.
  The third is possible by a compactness argument using the facts that
  $[d(x,y) \leq \varepsilon_i+2^{-i-2}\varepsilon]$ contains a neighbourhood of $[d(x,y)\leq\varepsilon_i]$, and  
  $$[d(x,y)\leq\varepsilon_i] = [d(x,y)\leq\varepsilon_i]^{\fp(0)} = \bigcap_{\delta>0}[d(x,y)\leq\varepsilon_i]^{\fp(\delta)}.$$
  Let us also agree that $\delta_{-1} = \infty$.

  We now choose a sequence $b_i \in M$ such that
  $\models p^{\fp(\varepsilon_i)}(b_i^{\delta_{i-1}},a^{\varepsilon_i})$:\\
  -- Since $\delta_{-1} = \infty$ and $p(x,a)$ is consistent, any $b_0 \in M$ will
  do.\\
  -- Let $b_i$ be given.
  Then in an elementary extension of $M$ there exists $c$ such that
  $d(c,b_i) \leq \delta_{i-1}$ and $\models p^{\fp(\varepsilon_i)}(c,a^{\varepsilon_i})$.
  Let $q(x,y,z) = \tp(c,b_i,a)$.
  By the saturation assumption there exists $b_{i+1} \in M$ such that
  $\models q^{\fp(\delta_i)}(b_{i+1}^{\delta_i},b_i^{\delta_i},a^{\delta_i})$.
  We know that $q(x,y,z) \vdash p^{\fp(\varepsilon_i)}(x,z^{\varepsilon_i})$,
  so:
  \begin{align*}
    & q^{\fp(\delta_i)}(x,y,z) \vdash p^{\fp(\varepsilon_i+\delta_i)}(x,z^{\varepsilon_i+2^{-i-2}\varepsilon})
    \vdash p^{\fp(\varepsilon_{i+1})}(x,z^{\varepsilon_i+2^{-i-2}})\\
    & q^{\fp(\delta_i)}(x^{\delta_i},y^{\delta_i},z^{\delta_i}) \vdash
    p^{\fp(\varepsilon_{i+1})}(x^{\delta_i},z^{\varepsilon_i+2^{-i-2}\varepsilon+\delta_i})
    \vdash  p^{\fp(\varepsilon_{i+1})}(x^{\delta_i},z^{\varepsilon_{i+1}}).
  \end{align*}
  Thus $\models p^{\fp(\varepsilon_{i+1})}(b_{i+1}^{\delta_i},a^{\varepsilon_{i+1}})$ as required.

  We also know that $q(x,y,z) \vdash d(x,y) \leq \delta_{i-1}$.
  It follows that $q^{\fp(\delta_i)}(x,y,z) \vdash d(x,y) \leq 2^{-i+1}$ (except
  when $i = 0$), so
  $d(b_i,b_{i+1}) \leq 2^{-i+1} + 2\delta_i \leq 2^{-i-1}(4+\varepsilon)$, so
  $(b_i\colon i \in \bN)$ is a Cauchy sequence in $M$ and therefore converges
  to some $b \in M$.
  For all $i < j \in \bN$ we have $\models p^{\fp(\varepsilon)}(b_j^{\delta_i},a^\varepsilon)$,
  so $\models p^{\fp(\varepsilon)}(b^{\delta_i},a^\varepsilon)$ for all $i \in \bN$, and as $\delta_i \to 0$
  we conclude that $\models p^{\fp(\varepsilon)}(b,a^\varepsilon)$, as required.
\end{proof}

\begin{prp}
  \label{prp:PertOSatRlsd}
  Assume that $M$ is $\fp$-approximately $\aleph_0$-saturated.
  Then for every finite tuple $\bar a \in M$, type
  $p(\bar x,\bar a) \in \tS_n(\bar a)$ and $\varepsilon > 0$ there are
  $\bar b,\bar a' \in M$ such that:
  \begin{enumerate}
  \item $d(\bar a,\bar a') \leq \varepsilon$.
  \item $\models p^{\fp(\varepsilon)}(\bar b,\bar a')$.
  \end{enumerate}
\end{prp}
\begin{proof}
  Let
  \begin{align*}
    q(\bar x,\bar y,\bar a) & := p(\bar x,\bar y) \wedge \bar y = \bar a.
  \end{align*}
  By Step II there are
  $\bar b, \bar a' \in M$ such that
  $\models q^{\fp(\varepsilon)}(\bar b,\bar a',\bar a^\varepsilon)$.
  Since $q(\bar x,\bar y,\bar z)$ implies that $\bar y = \bar z$,
  so does $q^{\fp(\varepsilon)}$, so $\models p(\bar b,\bar a')$ and
  $d(\bar a',\bar a) \leq \varepsilon$, as required.
\end{proof}

\begin{prp}
  \label{prp:PertOSatIsom}
  Any two elementarily equivalent separable $\fp$-approximately
  $\aleph_0$-saturated structures are $\fp$-isomorphic.
\end{prp}
\begin{proof}
  Let $M \equiv N$ be two separable $\fp$-approximately $\aleph_0$-saturated
  models, and let $\varepsilon > 0$ be given.
  Let $M_0 = \{a_i\colon i \in \bN\}$ and $N_0 = \{b_i\colon i \in \bN\}$ be
  countable dense subsets of $M$ and $N$, respectively.

  Define for convenience $\varepsilon_i = (1-2^{-i})\varepsilon$ for all $i \in \bN$.
  As $\fp(\varepsilon)$ is uniformly continuous, we may also choose
  $\delta_i > 0$ such that $[d(x,y) \leq \delta_i]^{\fp(\varepsilon)} \subseteq [d(x,y) \leq 2^{-i-1}]$
  (so in particular, $\delta_i \leq 2^{-i-1}$).

  We construct a sequence of mappings
  $f_i\colon A_i \to N$ and $g_i\colon B_i \to M$, where $A_i \subseteq M$ and
  $B_i \subseteq N$ are finite, such that:
  \begin{enumerate}
  \item $A_0 = B_0 = \emptyset$, and for $i > 0$:
    \begin{align*}
      A_{i+1} & = a_{\leq i} \cup A_i \cup g_i(B_i) \\
      B_{i+1} & = b_{\leq i} \cup B_i \cup f_{i+1}(A_{i+1}).
    \end{align*}
  \item For all $c \in A_i$: $d(c,g_i\circ f_i(c)) \leq \delta_i$.
  \item For all $c \in B_i$: $d(c,f_{i+1}\circ g_i(c)) \leq \delta_i$.
  \item For each $i$,
    $f_i$ is a $\fp(\varepsilon_{2i})$-perturbation and
    $g_i$ is a $\fp(\varepsilon_{2i+1})$-one.
  \end{enumerate}

  We start with $f_0 = \emptyset$, which is $0$-as we assume that
  $M \equiv N$.

  Assume that $f_i$ is given.
  Then $A_i$ is given, and is finite by the induction hypothesis, and
  this determines $B_i$ which is also finite.
  Fix enumerations for $A_i$ and $B_i$ as finite tuples, and let
  $p(\bar x,\bar y) = \tp^N(B_i,f(A_i))$.

  As $f_i$ is a $\fp(\varepsilon_{2i})$-perturbation, there is a type
  $q(\bar x,\bar y) \in p^{\fp(\varepsilon_{2i})}$ such that
  $q(\bar x,A_i)$ is consistent.
  By $\fp$-approximate $\aleph_0$-saturation of $M$ there are tuples
  $B_i', A_i' \subseteq M$ such that $d(A_i,A_i') \leq \delta_i$ and
  $M \models q(B_i',A_i')^{\fp(2^{-2i-1}\varepsilon)}$.
  Then $g_i\colon B_i \mapsto B_i'$ is $\fp(\varepsilon_{2i+1})$-elementary, so it will
  do.

  We construct $f_{i+1}$ from $g_i$ similarly.

  We now have for all $c \in A_i$:
  \begin{align*}
    d(c,g_i\circ f_i(c)) \leq \delta_i
    & \Longrightarrow d(f_{i+1}(c),f_{i+1}\circ g_i\circ f_i(c)) \leq 2^{-i-1} \\
    & \Longrightarrow d(f_{i+1}(c),f_i(c)) \leq 2^{-i}.
  \end{align*}
  Therefore the sequence of mappings $f_i$ converges to a mapping
  $f\colon A \to N$, where $A = \bigcup A_i$.
  As $f_i$ is an $\fp(\varepsilon)$-perturbation for all $i$ so is $f$.
  As $M_0 \subseteq A$ we have $\bar A = M$,
  so $f$ extends uniquely to a
  $\fp(\varepsilon)$-perturbation $\bar f\colon M \to N$.
  An $\fp(\varepsilon)$-perturbation $\bar g\colon N \to M$
  is constructed similarly.

  Finally, for $i < j \in \bN$ choose $k \geq j$ such that
  $2^{-k+2} \leq \delta_j$.
  Then:
  \begin{align*}
    d(a_i,\bar g\circ\bar f(a_i))
    & \leq d(a_i,\bar g\circ f_k(a_i)) + 2^{-j} \\
    & \leq d(a_i,g_{k+1}\circ f_k(a_i)) + 2^{-j+1} + 2^{-j} \\
    & \leq 2^{-j} + 2^{-j+1} + 2^{-j} \leq 2^{-j+2}.
  \end{align*}
  By letting $j \to \infty$ we see that $\bar g \circ \bar f$ is the identity
  on $M_0$, and therefore on $M$.
  Similarly $\bar f \circ \bar g = \id_N$.
\end{proof}

\section{Categoricty up to perturbation}
\label{sec:PertRN}

We now turn to the proof of a Ryll-Nardzewski style characterisation
of separable categoricity up so small perturbations.
As usual, $T$ denotes a theory and $\fp$ a perturbation system for
$T$.

\begin{dfn}
  Let $\kappa \geq |T|$ be a cardinal
  (recall that $|T| = |\cL| + \aleph_0$).
  We say that a theory $T$ is $\fp$-$\kappa$-categorical if it has a
  model of density character $\kappa$, and in addition every two
  models $M,N \models T$ of density character $\kappa$
  are $\fp$-isomorphic.
\end{dfn}
For the purpose of this definition we consider the density character
of a finite set to be $\aleph_0$.
In particular, the complete theory of a compact, or finite, structure,
will be considered $\aleph_0$-categorical, and therefore
$\fp$-$\aleph_0$-categorical for all $\fp$.

\begin{rmk}
  It is not difficult to verify a general converse to
  \fref{prp:PertOSatIsom}, i.e.,
  that if $\fp$ is a perturbation system
  and $M$ and $N$ are $\fp$-isomorphic then $M \equiv N$.
  Thus Vaught's Test holds just as well for perturbed categoricity:
  if $T$ has no compact models and is $\fp$-$\kappa$-categorical for some
  $\kappa \geq |T|$ then $T$ is complete.
\end{rmk}

\begin{conv}
  For the rest of this section we assume that $T$ admits non compact
  models.
\end{conv}

\begin{lem}
  A complete countable theory $T$ is $\fp$-$\aleph_0$-categorical if
  and only if all separable models of $T$ are $\fp$-approximately
  $\aleph_0$-saturated.
\end{lem}
\begin{proof}
  Right to left follows from \ref{prp:PertOSatIsom}.

  Conversely, assume that $T$ is $\fp$-$\aleph_0$-categorical, and let
  $M \models T$ be separable.
  Let $\bar a \in M^n$, $q(\bar x) = \tp(\bar a)$,
  and let $q(y,\bar a) \in \tS_1(\bar a)$,
  where $q(y,\bar x) \in \tS_{n+1}(T)$ is a complete
  pure type.
  Let also $\varepsilon > 0$, and $\delta = \delta_\rho(\varepsilon) > 0$.

  By the downward Löwenheim-Skolem theorem there exists a separable
  model $N \models T$ such that for every $n$-tuple $\bar b \in N$,
  if $\bar b$ satisfies $p^{\fp(\varepsilon/2)}$ then
  $q^{\fp(\varepsilon/2)}(y,\bar b^\delta)$ is realised in $N$.
  By assumption there exists a $\fp(\varepsilon/2)$-isomorphism
  $f\colon M \to N$.
  Then $\bar b = f(\bar a) \models p^{\fp(\varepsilon/2)}$,
  and let $c \in N$ be such that
  $\models q^{\fp(\varepsilon/2)}(c,\bar b^\delta)$.
  Letting $d = f^{-1}(c)$ we get
  $\models q^{\fp(\varepsilon)}(d,\bar a^\varepsilon)$.

  Thus $M$ is $\fp$-approximately $\aleph_0$-saturated.
\end{proof}

We observe that if $\fp$ is a perturbation system, then the
topology on $\tS_n(T)$ induced by $d_\fp$ is finer than the logic
topology.
Indeed, if $U$ is a neighbourhood of $p$, then
$\bigcap_{\varepsilon > 0} p^{\fp(\varepsilon)} = \{p\} \subseteq U$,
and by compactness we must have
$p^{\fp(\varepsilon)} \subseteq U$ for some $\varepsilon > 0$.
In fact,  $d_\fp$ is usually too fine to be used directly for
characterising $\aleph_0$-categoricity.
For example, in case
of the identity perturbation (i.e., no perturbation allowed at all),
$d_{\fp,n}$ is a discrete metric (with values in $\{0,\infty\}$), while the
standard Ryll-Nardzewski theorem for continuous logic does consider a
much coarser topology, namely that induced by the metric $d$.
We therefore need to take both metrics into account.

\begin{thm}
  \label{thm:PertRN}
  Let $T$ be a complete countable theory, $\fp$ a perturbation system
  for $T$.
  Then the following are equivalent:
  \begin{enumerate}
  \item The theory $T$ is $\fp$-$\aleph_0$-categorical.
  \item For every $n \in \bN$, finite $\bar a$,
    $p \in \tS_n(\bar a)$ and $\varepsilon > 0$,
    the set $[p^{\fp(\varepsilon)}(\bar x^\varepsilon,\bar a^\varepsilon)]$ has
    non empty interior in $\tS_n(\bar a)$.
  \item Same restricted to $n = 1$.
  \end{enumerate}
\end{thm}
\begin{proof}
  \begin{cycprf}
  \item[\impnext]
    Assume there is some finite tuple $\bar a$, $n \in \bN$ and
    $p(\bar x,\bar a) \in \tS_n(\bar a)$, such that
    for some $\varepsilon > 0$ the set
    $[p^{\fp(\varepsilon)}(\bar x^\varepsilon,\bar a^\varepsilon)]$ has
    empty interior in $\tS_n(\bar a)$.
    Then it is
    nowhere dense
    in $\tS_n(\bar a)$, and can be omitted in a dense subset
    of some separable model $(M,\bar a) \models T_{\bar a}$.
    Therefore $[p^{\fp(\varepsilon)}(\bar x,\bar a^\varepsilon)]$ is omitted in
    $M$, which is therefore not $\fp$-approximately $\aleph_0$-saturated.
    Therefore $T$ cannot be $\fp$-$\aleph_0$-categorical.
  \item[\impnext] Clear.
  \item[\impfirst]
    We show that every $M \models T$ is $\fp$-approximately
    $\aleph_0$-saturated.
    Indeed, let $\bar a \in M$ be a finite tuple and
    $p(x,\bar a) \in \tS_1(T)$.
    As $[p^{\fp(\varepsilon)}(x^\varepsilon,\bar a^\varepsilon)]$ has
    non empty interior in
    $\tS_n(\bar a)$ is must be realised in $M$.
  \end{cycprf}
\end{proof}

We may wish to combine the two metrics in a single one.
While one may try to achieve this through various general approaches
for the combination of two metrics, the specific situation in which we
find ourselves suggests a specific construction as the ``natural'' one.

Fix $m \in \bN$, and let
$\cL_{\bar c} = \cL \cup \{c_i\colon i <m\}$, where each $c_i$ is a new distinct
constant symbol.
Let $T_{\bar c}$ be the (incomplete) $\cL_{\bar c}$-theory generated by
$T$.
We extend $\fp$ into a perturbation system $\fp_{\bar c}$ for
$T_{\bar c}$ by
allowing the new constant symbols to move a little.
It is more convenient to think in terms of a relational language, in
which each of the constants $c_i$ is represented by a unary predicate
giving the distance to $c_i$.
We therefore define $\fp_{\bar c}^0(\varepsilon)$ to be the perturbation
pre-radius which, for $p,q \in \tS_n(T_{\bar c})$, allows to perturb $p$
to $q$ if and only if:
\begin{enumerate}
\item $q\rest_\cL \in (p\rest_\cL)^{\fp(\varepsilon)}$; and:
\item For all $i<m$ and $j<n$: $|d(x_j,c_i)^p - d(x_j,c_i)^q| \leq \varepsilon$.
\end{enumerate}
It is easy to verify that $\fp_{\bar c}^0$ is a uniformly continuous
perturbation
pre-system, which generates a perturbation system
$\fp_{\bar c} = \llbracket\fp^0_{\bar c}\rrbracket$.
Thus $\fp_{\bar c}$ can be roughly described as allowing to perturb
models of $T$ according to $\fp$, and to move the new constants
(i.e., change the distance to them) a little as well.
If $\bar c = \emptyset$ we changed nothing: $\fp_\emptyset = \fp$.

By definition of the bi-reduct $\llbracket\cdot\rrbracket$, we have for all $\varepsilon>0$, and
$(M,\bar a),(N,\bar b) \models T_{\bar c}$:
\begin{align*}
  & \BiPert_{\fp_{\bar c}^0(\varepsilon)}((M,\bar a),(N,\bar b)) = 
  \BiPert_{\fp_{\bar c}(\varepsilon)}((M,\bar a),(N,\bar b)) = \\
  & \qquad \left\{ f \in \BiPert_{\fp(\varepsilon)}(M,N)\colon
    (\forall e \in M, i<m)\bigl( |d^M(e,a_i) - d^N(f(e),b_i)| \leq \varepsilon \bigr) \right\}.
\end{align*}

The space $\tS_0(T_{\bar c})$ is the set of completions of
$T_{\bar c}$ and can be naturally identified with $\tS_m(T)$.
We define a metric $\tilde d_{\fp,m}$ on $\tS_m(T)$ as the image of
$d_{\fp_{\bar c}}$ under this identification.
Equivalently:
\begin{dfn}
  \label{dfn:JointMetric}
  For $p,q \in \tS_n(T)$, we define $\tilde d_\fp(p,q)$ as the infimum
  of all $\varepsilon$ for which there exist models $M,N \models T$,
  $\bar a \in M^n$ and $\bar b \in N^n$ and a mapping
  $f\colon M\to N$ such that:
  \begin{enumerate}
  \item $\bar a \models p$ and $\bar b \models q$.
  \item $f \in \BiPert_{\fp(\varepsilon)}(M,N)$.
  \item For all $i < n$ and $c \in M$:
    $|d^M(c,a_i) - d^N(f(c),b_i)| \leq \varepsilon$.
  \end{enumerate}
\end{dfn}

Alternatively, we may wish to restrict $\fp_{\bar c}$
to a specific completion of $T_{\bar c}$.
Any such completion is of the form $T_{\bar a} = \Th(M,\bar a)$,
where $M \models T$ and $\bar a \in M^n$.
Let us denote the restriction of $\fp_{\bar c}$ to $T_{\bar a}$ by
$\fp_{\bar a}$.

Of course, once we have constructed $\fp_{\bar a}$, we can construct
$\tilde d_{\fp_{\bar a}}$ as above, and
it follows immediately from the definitions that:
\begin{lem}
  \label{lem:JointMetParams}
  The construction $\fp \mapsto \tilde d_\fp$ commutes with the addition
  of parameters, in the sense that for all
  $\bar a$, $\bar b$ and $\bar c$, if $|\bar b| = |\bar c|$ then:
  \begin{align*}
    \tilde d_{\fp_{\bar a}}(\tp(\bar b/\bar a),\tp(\bar c/\bar a))
    =
    \tilde d_{\fp}(\tp(\bar b,\bar a),\tp(\bar c,\bar a)).
  \end{align*}
\end{lem}

In an arbitrary metric space $(X,d)$, let
$B_d(x,\varepsilon)$ denote the closed $\varepsilon$-ball around a point $x$.
The following result characterise the topology defined by
$\tilde d_\fp$:
\begin{lem}
  \label{lem:JointMetric}
  Fix $n \in \bN$ and a finite tuple $\bar a \in M \models T$.
  The metric $\tilde d_{\fp_{\bar a}}$ is coarser (i.e., smaller)
  on $\tS_n(\bar a)$ than both $d$
  and $d_{\fp_{\bar a}}$, and finer than the logic topology.

  Also, for every $p(\bar x,\bar a) \in \tS_n(\bar a)$,
  the family
  \hbox{$\bigl\{
    [p^{\fp(\varepsilon)}(\bar x^\varepsilon,\bar a^\varepsilon)]
    \colon \varepsilon > 0
    \bigr\}$}
  forms a base of
  $\tilde d_{\fp_{\bar a}}$-neighbourhoods for $p(\bar x,\bar a)$.
\end{lem}
\begin{proof}
  Let us start by showing that for every $\varepsilon > 0$ there is
  $\varepsilon' > 0$
  such that:
  \begin{align*}
    & [p^{\fp(\varepsilon')}(\bar x^{\varepsilon'},\bar a^{\varepsilon'})]
    \subseteq B_{\tilde d_{\fp_{\bar a}}}(p(\bar x,\bar a),\varepsilon)
    \subseteq [p^{\fp(\varepsilon)}(\bar x^\varepsilon,\bar a^\varepsilon)].
  \end{align*}
  Let us first consider the case without parameters.
  The set
  $\bigl\{
  (q(x,y),q'(x,y)) \in \tS_2(T)
  \colon |d(x,y)^q - d(x,y)^{q'}| \geq \varepsilon/2
  \bigr\}$
  is closed and disjoint of the diagonal, so by compactness there is
  $\varepsilon' > 0$ such for all $(q,q') \in \fp_2(\varepsilon')$:
  $|d(x,y)^q - d(x,y)^{q'}| \leq \varepsilon/2$.
  We may of course assume that $\varepsilon' \leq \varepsilon/2$, and the first
  inclusion follows.
  The second inclusion is immediate from the definition
  of $\tilde d_\fp$.

  The case over parameters $\bar a$ follows from the case without
  parameters and the fact that by
  \fref{lem:JointMetParams}:
  $$
  B_{\tilde d_{\fp_{\bar a}}}(p(\bar x,\bar a),\varepsilon)
  =
  \big\{ q(\bar x,\bar a) \in \tS_n(\bar a)\colon
  q(\bar x,\bar y) \in B_{\tilde d_\fp}(p,\varepsilon) \big\}.
  $$

  Finally, let $K \subseteq \tS_n(T)$ be closed in the logic topology and
  $q \notin K$.
  Then there is $\varepsilon > 0$ such that
  $q^{\fp(\varepsilon)} \cap K = \emptyset$, and since
  $q^{\fp(\varepsilon)}$ is closed in the logic topology there is also
  $\varepsilon' > 0$
  such that $[q^{\fp(\varepsilon)}(\bar x^{\varepsilon'})] \cap K = \emptyset$.
  Letting $\varepsilon'' = \min\{\varepsilon,\varepsilon'\}$ we see that
  $[q^{\fp(\varepsilon'')}(\bar x^{\varepsilon''})] \cap K = \emptyset$.
  Therefore $K$ is $\tilde d_\fp$-closed.
  This shows that $\tilde d_\fp$ refines the logic topology.
  It is clearly coarser than both $d$ and $d_\fp$.
  Substituting $T_{\bar a}$ for $T$ in the last argument we get the
  case with parameters.
\end{proof}

Thus we can restate \fref{thm:PertRN} as:
\begin{thm}
  Let $T$ be a complete countable theory.
  Then the following are equivalent:
  \begin{enumerate}
  \item The theory $T$ is $\fp$-$\aleph_0$-categorical.
  \item For every $n \in \bN$, finite $\bar a$,
    $p \in \tS_n(\bar a)$ and $\varepsilon > 0$,
    the $\varepsilon$-ball
    $B_{\tilde d_{\fp_{\bar a}}}(p,\varepsilon)$ has non empty
    interior in the logic topology on $\tS_n(\bar a)$.
  \item Same restricted to $n = 1$.
  \end{enumerate}
\end{thm}

The statement of the result in terms of non empty interior may sound a
little weird, as the non perturbed Ryll-Nardzewski theorem tells us
that $T$ is $\aleph_0$-categorical if and only if the metric $d$ coincides
with the logic topology.
In order to explain this apparent discrepancy let us make a few
more observations.

First, the coincidence of the logic topology with the metric
$\tilde d_\fp$ is a sufficient condition for $T$ to be
$\fp$-$\aleph_0$-categorical.
In this case it suffices to check
$\tS_n(T)$ alone (i.e., no need to consider parameters).
\begin{prp}
  \label{prp:PertRNSuff}
  Assume that $T$ is countable and complete, and
  $\tilde d_\fp$ coincides with the logic topology on
  $\tS_n(T)$ for all $T$.
  Then $T$ is $\fp$-$\aleph_0$-categorical.
\end{prp}
\begin{proof}
  Let $M \models T$, and let
  $\bar a \in M$, $p(x,\bar a) \in \tS_1(\bar a)$,
  and $\varepsilon > 0$.
  By assumption $B_{\tilde d_\fp}(p(x,\bar y),\varepsilon)$ is a
  neighbourhood of $p$, so there is a formula $\varphi(x,\bar y)$
  such that
  \begin{gather*}
    p \in [\varphi = 0]
    \subseteq [\varphi < 1/2]
    \subseteq B_{\tilde d_\fp}(p(x,\bar y),\varepsilon')
    \subseteq [p^{\fp(\varepsilon)}(x^\varepsilon,\bar y^\varepsilon)].
  \end{gather*}
  Therefore $[\varphi(x,\bar a) < 1/2]$ is a non empty open subset of
  $\tS_1(\bar a)$ (as it contains $p(x,\bar a)$),
  and is therefore realised in $M$.
  Thus $p^{\fp(\varepsilon)}(x^\varepsilon,\bar a^\varepsilon)$ is
  realised in $M$.
\end{proof}

We should point out that the consideration of parameters in
\fref{thm:PertRN} is unavoidable.
Indeed, if $p(x,\bar a) \in \tS_{1}(\bar a)$ and we only assume that
$B_{\tilde d_\fp}(p(x,\bar y),\varepsilon)$
has non empty interior in $\tS_{n+1}(T)$,
which need not necessarily contain
$p(x,\bar y)$, it may
happen that no type $q(x,\bar y)$ in this interior is consistent with
$r(\bar y)  = \tp(\bar a)$, so pulling up to $\tS_1(\bar a)$ we may
end up with an empty set.

The sufficient condition in \fref{prp:PertRNSuff} seems far
more convenient and natural than the one in \fref{thm:PertRN}, and one
might hope to show that it is also necessary.
The following example shows that this is impossible.
Roughly speaking, this example says that if the ``if and only
if'' variant of \fref{prp:PertRNSuff} were true, we could prove
Vaught's no-two-models theorem, which fails in
continuous first order logic.
\begin{exm}
  \label{exm:LpCntrExm}
  Let $T$ be the theory of atomless $L^p$-Banach lattices for some
  fixed $p \in [1,\infty)$,
  studied in \cite{BenYaacov-Berenstein-Henson:LpBanachLattices}.
  It is known that $T$ is $\aleph_0$-categorical, all of its separable
  models being isomorphic to $L^p[0,1]$.

  Let $f$ be any positive function of norm $1$ (this
  determines $\tp(f)$), say $f = \chi_{[0,1]}$, then
  $T_f$ has precisely two non isomorphic separable models, namely
  $(L^p[0,1],\chi_{[0,1]})$ and $(L^p[0,2],\chi_{[0,1]})$.
  (The theory $T_f$ is $\aleph_0$-categorical up to perturbations of
  the new constant $f$, but that's not what we are looking for).
  Let $g$ be another positive function of norm $1$ such that
  $f\wedge g = 0$ (this determines $\tp(f,g)$).
  Then again, $T_{f,g}$ has precisely two separable models,
  $(L^p[0,2],\chi_{[0,1]},\chi_{[1,2]})$ and
  $(L^p[0,3],\chi_{[0,1]},\chi_{[1,2]})$.

  Let $\fp$ be the identity perturbation system for $T_f$,
  and thus $\fp_g$ is the perturbation system for $T_{f,g}$ that
  allows to perturb $g$ while preserving all the rest untouched.
  Then the two models above are $\fp_g$-isomorphic, so
  $T_{f,g}$ is $\fp_g$-$\aleph_0$-categorical.
  Let $\pi_n\colon \tS_n(T_{f,g}) \to \tS_n(T_f)$ be the reduct projection.
  As $\fp$ is the identity perturbation on $T_f$,
  $\tilde d_\fp = d$ on $\tS_n(T_f)$.
  Therefore, if $U \subseteq \tS_n(T_f)$ is $d$-open then
  $\pi_n^{-1}(U) \subseteq \tS_n(T_{f,g})$ is $\tilde d_{\fp_g}$-open.

  But $T_f$ is not $\aleph_0$-categorical, so the metric $d$ defines a
  non compact topology on $\tS_n(T)$, whereby
  $\tilde d_{\fp_g}$ defines a non compact topology on
  $\tS_n(T_{f,g})$,
  which in particular cannot coincide with the logic topology, even 
  though $T_{f,g}$ is $\fp_g$-$\aleph_0$-categorical.
\end{exm}

One last point arises from a comparison of \fref{thm:PertRN} with the
unperturbed Ryll-Nardzewski Theorem for continuous logic.
The latter characterises unperturbed $\aleph_0$-categoricity
by the coincidence of the logic topology with the metric, and thus does
not seem to be follow as a special case of \fref{thm:PertRN}.
To see that it actually does, we need to explore some further
properties perturbation metrics may have.

Let us start by recalling properties of the standard metric $d$ on
$\tS_n(T)$.
We observe in \cite{BenYaacov-Usvyatsov:CFO}
that the metric $d$ has the following
properties:
\begin{enumerate}
\item It refines the logic topology.
\item If $F = [p(\bar x)] \subseteq \tS_n(T)$ is closed, then so is
  $F^\varepsilon = \{p\colon d(p,F) \leq \varepsilon\} = [p(\bar x^\varepsilon)]$.
\item For every injective $\sigma\colon n\to m$, $p \in \tS_n(T)$,
  $q \in \tS_m(T)$:
  $$d(p,\sigma^*(q)) = d({\sigma^*}^{-1}(p),q).$$
\end{enumerate}
A perturbation metric has all these properties as well, and in fact
satisfies the last one also for $\sigma$ which is not injective.
One last interesting property of $(\tS_n(T),d)$ is analogous to
the second property:
\begin{lem}
  \label{lem:DistIsOpen}
  If $U \subseteq \tS_n(T)$ is open, then so is
  $U^{<\varepsilon} = \{p\colon d(p,U) < \varepsilon\}$.
\end{lem}
\begin{proof}
  It suffices to show this for a basis of open sets, i.e., for sets of
  the form $U = [\varphi(\bar x)<\varepsilon]$.
  But then
  $U^{<\varepsilon} = [\inf_{\bar y} (\varphi(\bar y) \vee d(\bar x,\bar y)) < \varepsilon]$ is open.
\end{proof}
For lack of a better name, let us call provisionally a metric on a
topological space \emph{open} if it satisfies the property of
\fref{lem:DistIsOpen}.

\begin{dfn}
  Let $\fp$ be a perturbation system for $T$.
  \begin{enumerate}
  \item We say that $\fp$ is  \emph{open} if $d_\fp$ is open on
    $\tS_n(T)$ for all $n$.
  \item We say that $\fp$ is \emph{weakly open} if for all
    $\varepsilon > 0$
    and $n \in \bN$ there is $\delta > 0$ such that for every open set
    $U \subseteq \tS_n(T)$:
    \begin{gather*}
      U^{\tilde d_\fp<\delta}
      \subseteq \bigl( U^{\tilde d_\fp < \varepsilon} \bigr)^\circ.
    \end{gather*}
    (Where
    $U^{\tilde d_\fp<\delta}
    = \{p\colon \tilde d_\fp(p,U) < \delta\}$.)
  \end{enumerate}
\end{dfn}

\begin{lem}
  Let $\fp$ be a perturbation system.
  \begin{enumerate}
  \item $\fp$ is weakly open if and only if for every
    $\varepsilon > 0$ and $n \in \bN$ there is $\delta > 0$ such that
    for every open $U \subseteq \tS_n(T)$:
    \begin{gather*}
      U^{\fp(\delta)}
      \subseteq \bigl( U^{\tilde d_\fp < \varepsilon} \bigr)^\circ.
    \end{gather*}
  \item If $\fp$ is open then it is weakly open.
  \item If $\tilde d_\fp$ is open on $\tS_n(T)$ for all $n$ then $\fp$
    is weakly open.
  \end{enumerate}
\end{lem}
\begin{proof}
  \begin{enumerate}
  \item For one direction use the fact that
    $U^{\fp(\delta/2)} \subseteq U^{\tilde d_\fp<\delta}$.
    For the other, assume that
    $U^{\fp(\delta)}
    \subseteq \bigl( U^{\tilde d_\fp < \varepsilon/2} \bigr)^\circ$
    and
    $\delta < \varepsilon/2$.
    Then since the metric $d$ is open:
    \begin{gather*}
      U^{\tilde d_\fp<\delta}
      \subseteq
      (U^{\fp(\delta)})^{d<\delta}
      \subseteq
      \left(
        \bigl( U^{\tilde d_\fp < \varepsilon/2} \bigr)^\circ
      \right)^{d<\delta}
      \subseteq
      \left(
        (U^{\tilde d_\fp < \varepsilon/2})^{d<\delta}
      \right)^\circ
      \subseteq
      \bigl( U^{\tilde d_\fp < \varepsilon} \bigr)^\circ.
    \end{gather*}
  \item We use the criterion from the previous item:
    \begin{gather*}
      U^{\fp(\varepsilon/2)}
      \subseteq
      U^{d_\fp<\varepsilon}
      =
      \bigl( U^{d_\fp<\varepsilon} \bigr)^\circ
      \subseteq
      \bigl( U^{\tilde d_\fp<\varepsilon} \bigr)^\circ.
    \end{gather*}
  \item Immediate from the definition.
    \qedhere
  \end{enumerate}
\end{proof}

\begin{thm}
  \label{thm:OpenPertRN}
  Let $T$ be a complete countable theory, $\fp$ a weakly open
  perturbation system.
  Then $T$ is $\fp$-$\aleph_0$-categorical if and only if
  for every $n$, $\tilde d_{\fp}$ coincides with the logic topology on
  $\tS_n(T)$.
\end{thm}
\begin{proof}
  Right to left is by \fref{prp:PertRNSuff}, so we prove left to
  right.

  Assume that $T$ is $\fp$-$\aleph_0$-categorical.
  Fix $p \in \tS_n(T)$ and $\varepsilon > 0$.
  Then by definition there is $\delta > 0$ such that for every open set
  $U \subseteq \tS_n(T)$:
  $U^{\tilde d_\fp < \delta}
  \subseteq \bigl( U^{\tilde d_\fp < \varepsilon/2} \bigr)^\circ$.
  We may also assume that $\delta < \varepsilon$.

  Let $U = B_{\tilde d_\fp}(p,\delta/2)^\circ$, so
  $U \neq \emptyset$ by \fref{thm:PertRN}.
  Then:
  \begin{gather*}
    p
    \in U^{\tilde d_\fp < \delta}
    \subseteq
    \bigl( U^{\tilde d_\fp < \varepsilon/2} \bigr)^\circ
    \subseteq
    \left(
      B_{\tilde d_\fp}(p,\delta/2)^{\tilde d_\fp < \varepsilon/2}
    \right)^\circ
    \subseteq
    \bigl( B_{\tilde d_\fp}(p,\varepsilon) \bigr)^\circ
  \end{gather*}
  Therefore $B_{\tilde d_\fp}(p,\varepsilon)$ is a logic neighbourhood of $p$
  for all $p$ and $\varepsilon > 0$.
  Since $\tilde d_\fp$ refines the logic topology, they must
  coincide.
\end{proof}

\begin{exm}
  The identity perturbation system is open.
\end{exm}

\begin{cor}[Henson's unperturbed Ryll-Nardzewski Theorem]
  A complete countable theory $T$ is $\aleph_0$-categorical if and
  only if the standard metric $d$ coincides with
  the logic topology on $\tS_n(T)$, for all $n \in \bN$.
\end{cor}
\begin{proof}
  Since $\tilde d_{\id} = d$.
\end{proof}

\begin{exm}
  Let $\bar a \in M \models T$, and let $p = \tp(\bar a)$ be isolated
  (i.e., $d(\bar x,p)$ is a definable predicate).
  Let $\fp$ be the identity perturbation for $T$, and $\fp_{\bar a}$
  as above be a perturbation system for $T_{\bar a}$ allowing to move
  the named parameter.
  Then $\fp_{\bar a}$ is open.
\end{exm}
\begin{proof}
  Exercise.
\end{proof}

Of course, in \fref{exm:LpCntrExm} the type of the new parameter
$\tp(g/f)$ was not isolated.

\section{Perturbations of automorphisms}
\label{sec:PertAut}

We conclude with a few problems concerning perturbations of
automorphisms which motivated the author's initial interest in
perturbations, and which the author therefore finds worthy of future
study.

One such problem comes from the study of the properties of
automorphism groups of classical (i.e., discrete) countable
structures, and in particular of ones whose first order theory is
$\aleph_0$-categorical, viewed as topological groups.
Model-theoretic questions of this kind are treated, for example, in
\cite{Hodges-Hodkinson-Lascar-Shelah:SmallIndex},
while more topologically profound questions are studied by
Kechris and Rosendal \cite{Kechris-Rosendal}.
It is natural to ask whether such
of these results can be generalised to the automorphism groups of
separable continuous structures (with a separably categorical theory).
A very simple instance would to consider the unitary group $U(H)$
where $H$ is a separable infinite dimensional Hilbert space.
Indeed, $U(H)$ is a polish group in the point-wise convergence
topology, also known as the strong operator topology, but it is
quickly revealed that $U(H)$ is just way too big for any of the
properties that Kechris and Rosendal were looking for
(e.g., existence of ample generics) to hold.

This is definitely not a new phenomenon.
We already know that the type space of a continuous theory, viewed as
a pure topological space, is too big.
In order to study
notions such as superstability, $\aleph_0$-stability, or even
local $\varphi$-stability, one needs to take an additional metric
structure into account, considering points (types) up to small
distance.
Such considerations date as far back as
Iovino's definition of $\lambda$-stability and of
superstability in Banach space structures \cite{Iovino:StableBanach}.
Similarly, the automorphism group of a metric structure admits a
natural metric, namely the metric of uniform convergence
(for example, the operator norm on $U(H)$ is the metric of
uniform convergence on the unit ball).
In the terminology of
\cite{BenYaacov:TopometricSpacesAndPerturbations},
$\Aut(M)$ is a \emph{topometric group}, namely a topological group
(in the point-wise convergence topology) which is at the same time a
metric group (in the uniform convergence metric),
such that in addition the distance function $d\colon G^2 \to \bR^+$
is lower semi-continuous in the topology.
One can then restate the question of the
existence of ample generics as follows:
\begin{qst}
  Let $M$ be a separable metric structure, $G = \Aut(M)$.
  Under what assumptions on $M$ can we find, for each $n \in \bN$,
  a tuple $\bar g \in G^n$ such that for every $\varepsilon > 0$,
  the $G$-conjugacy class of the (metric) $\varepsilon$-ball around
  $\bar g$ is (topologically) co-meagre?
  In other words can we find $\bar g$ such that the metric closure of
  the orbit of $\bar g$ is co-meagre?
  In particular, can one prove this is the case if $\Th(M)$ is
  $\aleph_0$-categorical and $\aleph_0$-stable?
\end{qst}
Considering an automorphism $\tau \in \Aut(M)$ up to small distance in
uniform convergence is essentially the same as considering
the structure $(M,\tau)$ up to a small perturbation of $\tau$,
whence the connection with the topic of the present paper.
For the special case of $U(H)$, a positive answer essentially follows
from \cite[Theorem~II.5.8]{Davidson:CsAlgebrasByExample}.
What about the automorphism group of the unique separable atomless
probability algebra?

Another question leading to similar considerations is raised by
Berenstein and Henson
\cite{Berenstein-Henson:ProbabilityWithAutomorphism}.
In this paper they consider
the theory of probability algebras with a generic automorphism, and
ask whether it is superstable
(equivalently, supersimple, since they showed that the theory is
stable).
In classical first order logic the answer would be positive, by a
theorem of Chatzidakis and Pillay
\cite{Chatzidakis-Pillay:GenericStructuresAndSimpleTheories}.
Henson's and Berenstein's was question was nonetheless answered
negatively by the author, raising the following natural
``second best'' question, namely whether the theory of probability algebras
with a generic automorphism is superstable
\emph{up to small perturbations of the automorphism}.
This was subsequently answered positively by the author and Berenstein
\cite{BenYaacov-Berenstein:HilbertProbabilityAutmorphismPerturbation},
where we show moreover that up to perturbations of the automorphism,
the theory is $\aleph_0$-stable.

\begin{qst}
  Let $T$ be a superstable continuous theory.
  Let
  \begin{gather*}
    T_\sigma = T \cup \{\text{``$\sigma$ is an automorphism''}\}.
  \end{gather*}
  Assume furthermore that $T_\sigma$ has a model companion $T_A$.
  Is $T_A$ supersimple up to small perturbations of $\sigma$?
\end{qst}
And in fact,
\begin{qst}
  What should it mean precisely
  for a theory to be supersimple up to small
  perturbations?
\end{qst}

Regarding the last question it should be pointed out that there are
several natural candidates for the definition of
``$a$ is independent up to distance $\varepsilon$ from $B$ over $A$''
(denoted usually $a^\varepsilon \ind_A B$).
While these notions of approximate independence are not equivalent,
they all give rise to the same notion of supersimplicity (see for
example in \cite{BenYaacov:SuperSimpleLovelyPairs}),
and in a stable theory they are further equivalent to superstability.
Superstability and $\lambda$-stability up to
perturbation are introduced by the author in
\cite{BenYaacov:TopometricSpacesAndPerturbations},
and one should seek a notion of supersimplicity up to perturbation
which, in stable theories, coincides with
superstability up to perturbation.

\bibliographystyle{amsalpha}
\bibliography{begnac}

\end{document}